\documentclass[12pt]{amsart}
\usepackage{url}

\usepackage{amsmath} 
\usepackage{amsfonts}
\usepackage{amssymb}
\usepackage{amsthm}

\usepackage{url}

\usepackage{latexsym}
\usepackage{color}
\usepackage{enumerate}
\usepackage{mathrsfs}
\usepackage{comment}
\usepackage{setspace}
\usepackage{ulem}

\newtheorem{theorem}{Theorem}[section]

\theoremstyle{definition}
\newtheorem{definition}[theorem]{Definition}

\newtheorem{conjecture}{Conjecture}[section]
\theoremstyle{remark}
\newtheorem{prop}{Proposition}[section]
\newtheorem{remark}[theorem]{Remark}

\numberwithin{equation}{section}
\def\R{\mathbb{R}}
\def\({\left(}
\def\){\right)}
\newcommand{\be}{\begin{equation}}
\newcommand{\ee}{\end{equation}}
\def\Ric{\textmd{Ric}}
\def\m{\mathfrak{m}}
\def\p{\partial}

\begin{document}
	
	\title[Minkowski-like inequality for static manifolds]{On a Minkowski-like inequality for asymptotically flat static manifolds}
	
	\author[McCormick]{Stephen McCormick}
	
	\address{Institutionen f\"{o}r Matematik, Kungliga Tekniska H\"{o}gskolan, 100 44 Stockholm, Sweden.}
	\curraddr{Matematiska institutionen, Uppsala universitet, 751 06 Uppsala, Sweden.}
	\email{stephen.mccormick@math.uu.se}

	
	\date{}
	
	\dedicatory{}
	
	\commby{}
	
	\maketitle 
	
	\begin{abstract}
		The Minkowski inequality is a classical inequality in differential geometry, giving a bound from below, on the total mean curvature of a convex surface in Euclidean space, in terms of its area. Recently there has been interest in proving versions of this inequality for manifolds other than $\R^n$; for example, such an inequality holds for surfaces in spatial Schwarzschild and AdS-Schwarzschild manifolds. In this note, we adapt a recent analysis of Y. Wei to prove a Minkowski-like inequality for general static asymptotically flat manifolds.
	\end{abstract}
	
	\section{Introduction}

The Minkowski inequality is a celebrated result in classical differential geometry, bounding the total mean curvature of a closed convex hypersurface $\Sigma$ in $\R^n$ from below in terms of its area \cite{Minkowski}. Precisely, we have
\begin{equation} \label{eq-mink}
\frac{1}{(n-1)\omega_{n-1}}\int_\Sigma H\,d\sigma\geq\( \frac{|\Sigma|}{\omega_{n-1}} \)^{\frac{n-2}{n-1}},
\end{equation}
where $|\Sigma|$ is the area of $\Sigma$ and $\omega_{n-1}$ is the area of the unit $(n-1)$-sphere. Moreover, equality holds if and only if $\Sigma$ is a round sphere. The hypotheses of \eqref{eq-mink} have since been improved to include mean convex star-shaped \cite{GMTZ,G-L09} and outer-minimizing \cite{H09} surfaces.

Recall, a closed surface $\Sigma$ is said to be outer-minimizing if it minimises area among all surfaces enclosing $\Sigma$. Note that the inequality (\ref{eq-mink}) has been transposed from how it is perhaps more commonly expressed, for a more direct comparison to (\ref{eq-main}) below.\\

In recent years, there has been interest in generalisations of the Minkowski inequality to surfaces embedded in manifolds other than Euclidean space. For example, Minkowski inequalities are known for surfaces in hyperbolic space \cite{LimaG,BHW}, Schwarzschild manifolds \cite{BHW,Wei2017}, and Schwarzschild-AdS manifolds \cite{BHW}. In this note, we prove a Minkowski inequality for general static asymptotically flat manifolds, generalising the classical inequality and the known inequality for Schwarzschild manifolds. As with several other proofs of Minkowski inequalities, our proof relies on monotonicity of a quantity under inverse mean curvature flow. The key contribution here is the observation that using the weak formulation of inverse mean curvature flow allows one to prove the inequality on a general static manifold, rather than working only within a fixed manifold where the smooth flow is known to be well-behaved. The proof adapts an analysis of Wei \cite{Wei2017} that is used to prove a Minkowski inequality for outer-minimizing surfaces in the Schwarzschild manifold.

First, we recall some definitions.

\begin{definition}
	A Riemannian manifold $(M,g)$ is said to be asymptotically flat (with one end) if $M$ minus a compact set is diffeomorphic to $\R^n$ minus a closed ball, the scalar curvature is integrable, ie. $R(g)\in L^1(M)$, and near infinity, $g$ satisfies:
	\begin{equation}
	g=\delta+O(|x|^{-\tau}),\qquad \partial g=O(|x|^{-\tau-1}),\qquad \partial^2 g=O(|x|^{-\tau-2})
	\end{equation}
	where $\delta$ is the flat metric and $\tau\in(1/2,1]$.
\end{definition}
It is well-known that this decay is sufficient to ensure that the ADM mass is well-defined.
\begin{definition}[\cite{ADM}]
	Let $(M,g)$ be an asymptotically flat manifold of dimension $n$ and one end. The ADM mass of $(M,g)$ is then given by
	\be 
	m=\frac{1}{2(n-1)\omega_{n-1}}\lim_{r\to\infty}\int_{S_r}\partial_i g_{ij}-\partial_j g_{ii}\,dS^j,
	\ee 
	where the limit is taken over spheres $S_r$ of radius $r$ in the asymptotic end, the coordinates near infinity are those coming from the usual Cartesian coordinates in Euclidean space, and repeated indices are summed over.\\
	
	Note that we use large spheres for the sake of convenience, but it is now well-known that the definition is independent of the limiting surfaces used. Furthermore, Cartesian coordinates are used to simplify the expression; however, the ADM mass is indeed a geometric quantity, independent of coordinates \cite{Bartnik-86,Chrusciel}.
\end{definition}

\begin{definition}\label{def-static}
	Let $(M,g)$ be an $n$-dimensional Riemannian manifold. A function $f$ on $M$ is called a static potential if it solves
	\be \label{eq-static}
	\Delta_g (f)g-\nabla^2_g f+f\Ric_g=0.
	\ee 
	If $(M,g)$ admits a positive solution to (\ref{eq-static}) then we say it is a static manifold.
\end{definition}

Note that if $(M,g)$ is asymptotically flat, then a well-known result of Corvino \cite{Corvino} implies that $g$ is scalar-flat, $R=0$. This, in turn, implies that a static potential satisfies $\Delta_g f=0$.

Throughout, we will always work with bounded static potentials, which at least in dimension 3, is implied by the assumption of positivity (see, for example, Proposition 3.1 of \cite{M-T}).

The terminology `static' comes from general relativity, and indeed so does the motivation for studying such manifolds. A manifold $(M,g)$ is static, with static potential $f$, if and only if the Lorentzian warped product metric
\be \label{eq-warpedproduct}
h=-f^2dt^2+g
\ee 
satisfies the vacuum Einstein equations. The metric (\ref{eq-warpedproduct}) is then a static spacetime in the sense of general relativity.\\

The main result of this note is the following inequality.

\begin{theorem}\label{thm-main}
	Let $(M,g)$ be an asymptotically flat manifold of dimension $3\leq n\leq 7$ with bounded positive static potential $f$. Assume $\partial M$ is not empty, and let $\Sigma$ be a connected component of the boundary that is outer-minimizing with (inward) mean curvature $H$. Assume further that any remaining components of the boundary are closed minimal surfaces.
	
	Then, after rescaling so that $f$ is asymptotic to $1$, we have
	\be \label{eq-main}
	\frac{1}{(n-1)\omega_{n-1}}\int_\Sigma fH dS\geq \( \frac{|\Sigma|}{\omega_{n-1}} \)^{\frac{n-2}{n-1}} -2m.
	\ee 
	Furthermore, we have equality if and only if $M$ has no boundary other than $\Sigma$ and the foliation of $M$ given by IMCF starting at $\Sigma$ is totally umbilic.
\end{theorem}

\begin{remark}
	As the manifolds we consider here have an interior boundary, the positive mass theorem does not apply. However, we do not require an assumption on the sign of $m$. For example, Theorem \ref{thm-main} holds when $(M,g)$ is taken to be an exterior region in a Schwarzschild manifold with negative mass.
\end{remark}

\begin{remark}
	In three dimensions, Huisken and Ilmanen's proof of the Riemannian Penrose inequality \cite{HI01} shows that the Hawking mass of an outer-minimizing surface bounds the ADM mass from below. This inequality can be expressed at
	\be 
	\frac{1}{16\pi}\( \frac{|\Sigma|}{4\pi}\)^{1/2}\int_{\Sigma} H^2\,dS\geq \( \frac{|\Sigma|}{4\pi}\)^{1/2}-2m.
	\ee 
	In higher dimensions, a recent result of Miao and the author \cite{MM16} gives a related inequality.	It would be interesting to compare these inequalities to \eqref{eq-main}.
\end{remark}

\section*{Acknowledgements}

{\sl The author would like to thank Yong Wei for many useful discussions in the preparation of this note, as well as to thank the Knut and Alice Wallenberg Foundation for financial support.}
\section{Properties of static potentials}
In the statement of Theorem \ref{thm-main}, we consider only static potentials that are positive and bounded. In this section we recall some basic properties of static potentials, and in particular, illustrate that this condition is natural and indeed such static manifolds are of particular interest to consider. Readers who are familiar with static metrics and Bartnik's quasi-local mass can skip this section.

Most of the literature pertaining to static manifolds considers the case $n=3$, motivated by general relativity. In particular, the study of static manifolds is very closely related to Bartnik's quasi-local mass. A common formulation of the Bartnik mass is the following, where usually one assumes $n=3$:
\begin{definition}
	Let the triple $(\Sigma,g,H)$ be a Riemannian metric $g$ on a closed $(n-1)$-manifold $\Sigma$, and $H$ be a positive function on $\Sigma$. Let $\mathcal{PM}(\Sigma,g,H)$ be the set of asymptotically flat manifolds with non-negative scalar curvature, boundary isometric to $(\Sigma,g)$ with (inward) mean curvature $H$, containing no closed minimal surfaces.\\
	
	The Bartnik mass is then given by
	\be \label{eq-bartnikmass}
	\m_B(\Sigma,g,H):=\inf_{(M,h)\in\mathcal{PM}(\Sigma,g,H)}\{ \m_{ADM}(M,h) \}.
	\ee
\end{definition}
While this mass seems effectively impossible to directly compute, Bartnik made the following conjecture\footnote{In dimension $n=3$.} regarding static metric extensions:
\begin{conjecture}
	Given $(\Sigma,g,H)$ as above, there is a unique static manifold in $\mathcal{PM}(\Sigma,g,H)$ and this manifold realises the infimum in (\ref{eq-bartnikmass}).
\end{conjecture}
\begin{remark}
	There are many subtly different formulations of Bartnik's quasi-local mass and the static metric extensions conjecture. For the sake of presentation, we avoid these technicalities here. In some cases, the conjecture has been resolved in the positive. For example, in the case of $(\Sigma,g,H)$ close to the data induced on a round sphere in $\R^3$, the existence of a static extension was proven by Miao \cite{Miao-CMP}. Recent work by Anderson \cite{Anderson} proves, given data $(\Sigma,g,H>0)$ then for sufficiently small $\lambda>0$, a static extension of $(\Sigma,g,\lambda H)$ can be found.
\end{remark}

It is worth reiterating, that while various static uniqueness theorems have been established, in the case of a manifold with mean convex boundary there exists a large class of examples of static manifolds, which are of particular interest in relation to Bartnik's quasi-local mass.

The following property of static potentials is very recent result of Huang, Martin and Miao \cite{HMM}, reformulated slightly such that it is directly applicable for us here.
\begin{theorem}[Theorem 1 of \cite{HMM}]\label{thm-fzero}
	Let $\Sigma$ be a closed minimal surface in an asyptotically flat manifold $(M,g)$ of dimension $3\leq n\leq7$, admitting a static potential $f$. Suppose there are no other closed minimal surfaces containing $\Sigma$, then $f\equiv0$ on $\Sigma$.\footnote{With the addition of a minor caveat, this result is in fact valid for all dimensions greater than $2$. However, as we are only concerned with $3\leq n\leq7$ here, we omit this caveat for the sake of clarity.}
\end{theorem}
The above theorem is used in the proof of our main inequality in order to allow for manifolds with boundary components outside of the one we flow from, provided they are minimal surfaces. We also make use of the asymptotics of an arbitrary static extension (See, for example, \cite{AM,M-T}). Specifically, if $f$ is a bounded positive static potential then it can be rescaled by a constant, so that at infinity $f$ has the expansion
\be f=1-\frac{m}{|x|^{n-2}}+o(|x|^{2-n}), \label{eq-expansion}\ee 
where $m$ is the ADM mass of $(M,g)$.

Throughout, we will assume a bounded static potential has always been rescaled appropriately so it has the form given by (\ref{eq-expansion}).

Given such a static potential, we can obtain the ADM mass from this asymptotic expansion via
\be \label{eq-m-integral}
\lim_{r\rightarrow\infty}\int_{S_r}\nabla f\cdot\nu d\mu_t=\lim_{r\rightarrow\infty}\int_{S_r}\frac{(n-2)m}{r^{n-1}} dS=(n-2)\omega_{n-1}m.
\ee

We end this section by quoting the following proposition, which is a result of Miao--Tam \cite{M-T} showing that the assumptions we impose on the static potential are satisfied for any static asymptotically Schwarzschildean 3-manifold.
\begin{prop}\label{prop-MiaoTam}
	Let $(M, g)$ be an asymptotically Schwarzschildean 3-manifold with nonzero mass and static potential $f$. Then outside a compact set, $f$ is bounded and does not change sign; ie. we can choose $f>0$.
\end{prop}

\section{Proof of the main theorem}
The proof of Theorem \ref{thm-main} follows the main proof of \cite{Wei2017} very closely, replacing the Schwarzschild potential with a general static potential. The inequality follows from the monotonicity of a quantity $Q(t)$ (defined below) under weak inverse mean curvature flow (IMCF). The classical (smooth) IMCF is a family of hypersurfaces $\Sigma_t$ given by $x:\Sigma\times[0,T)\to M$, that evolve with speed proportional to the reciprocal of the mean curvature:
\be 
\frac{\p x}{\p t}=\frac1H\nu,
\ee 
where $\nu$ is the unit normal pointing towards infinity. In general, the flow does not remain smooth for all time and one must work with a weak formulation of IMCF, which appropriately jumps past times where the flow fails to be smooth. As we do not require the technical aspects of weak IMCF in this note, we omit the details and refer the reader to \cite{HI01} for an excellent exposition.

For a weak solution to IMCF, $\Sigma_t$, and a given bounded, positive static potential $f$, we define on each $\Sigma_t$ the quantity
\be \label{eq-Qdef}
Q(t):=|\Sigma_t|^{-\frac{n-2}{n-1}} \(2(n-1)\omega_{n-1}m+\int_{\Sigma_t} fH\,d\mu_t \).
\ee 
We show that this is monotone along the weak IMCF. This monotonicity has been used previously (eg. \cite{BHW,Wei2017}) to prove Minkowski inequalities in Schwarzschild and Schwarzschild-AdS manifolds.\\

We consider the flow in a manifold $M$ with possibly disconnected boundary, commencing the flow from a chosen boundary component, $\Sigma$. The idea is, that commencing the flow from $\Sigma$, the weak flow `jumps' over the other components of the boundary and continues to flow afterwards. Similar to the proof of the Riemannian Penrose inequality using IMCF, we must assume the other boundary components are minimal surfaces in order to preserve the monotonicity across the jump. The analysis here closely follows that of Wei \cite{Wei2017} in the Schwarzschild case. We first consider the monotonicity of $Q(t)$ where the flow is smooth. The proof is essentially that of Brendle--Hung--Wang \cite{BHW} where the monotoncity is used to prove a Minkowski inequality in the AdS-Schwarzschild manifold; we include the details here as it is useful to illustrate the minor differences between the Schwarzschild case and a general static manifold.

\begin{prop} \label{prop-smoothmonotone}
	Let $\Sigma_t$ be a smooth solution to IMCF for $0<t_1<t_2<T$ on an asymptotically flat manifold $M$ with static potential $f$, satisfying the hypotheses of Theorem \ref{thm-main}. Then 
	\be 
	Q(t_2)\leq Q(t_1)
	\ee 
	with equality if and only if for each $t\in[t_1,t_2]$, $\Sigma_t$ is totally umbilic and $M$ has no boundary components outside of $\Sigma_{t_1}$.
\end{prop}
\begin{proof}
	The evolution equations for the mean curvature and the volume form of $\Sigma_t$ under IMCF are well-known (see, for example, \cite{HI01}) to be given by
	\be 
	\frac{\partial}{\partial t}H=-\Delta_{\Sigma_t}\frac{1}{H}-\frac{1}{H}\( |A|^2+Ric(\nu,\nu) \),
	\ee 
	and
	\be 
	\frac{\partial}{\partial t}d\mu_t=d\mu_t.
	\ee
	It is therefore straightforward to compute
	\begin{align*}
	\frac{\partial}{\partial t}\int_{\Sigma}fHd\mu_t&=\int_\Sigma\(\frac{\partial f}{\partial t}H+f\frac{\partial H}{\partial t}+fH\)d\mu_t\\
	&=\int_{\Sigma}\( \frac{1}{H}\nu\cdot\nabla (f)H-f\Delta_{\Sigma_t}\frac1H-\frac{f}{H}\(|A|^2+\Ric(\nu,\nu)\)+fH\)d\mu_t\\
	&\leq\int_\Sigma\( \nabla f\cdot \nu-\frac1H\( \Delta_{\Sigma_t}f+f\Ric(\nu,\nu) \)+\frac{n-2}{n-1}fH\)d\mu_t.
	\end{align*}
	where we have used the inequality $(n-1)|A|^2\geq H^2$,
	
	It follows that we have
	\be 
	\frac{\partial}{\partial t}\int_{\Sigma}fHd\mu_t\leq \int_\Sigma\( \nabla f\cdot \nu-\frac1H\( \Delta_{\Sigma_t}f+f\Ric(\nu,\nu) \)+\frac{n-2}{n-1}fH\)d\mu_t .
	\ee 
	From the condition that $f$ is a static potential \eqref{eq-static}, and the identity
	$$ \Delta_{\Sigma_t}f= \Delta f-\nabla^2 f(\nu,\nu)-H\nu\cdot\nabla f,$$
	we obtain
	\be \label{eq-Sigma-Laplace}
	\Delta_{\Sigma_t}f+f\Ric(\nu,\nu)=-H\nu\cdot \nabla f.
	\ee

	Then substituting (\ref{eq-Sigma-Laplace}) into the evolution eq gives
	\be \label{eq-evo1}
	\frac{\partial}{\partial t}\int_{\Sigma}fHd\mu_t\leq \int_\Sigma\( \frac{n-2}{n-1} fH+2\nabla f\cdot\nu \)d\mu_t,
	\ee 
	with equality if and only if $(n-1)|A|^2=H^2$; that is, $\Sigma_t$ is umbilic. Now, making use of the fact that $f$ is harmonic, the divergence theorem allows us to write the integral $\int_\Sigma \nabla f\cdot\nu d\mu_t$ as a surface integral at infinity, and a surface integral on (the possibly empty or disconnected) remaining boundary $\widehat{\partial M}:=\partial M\setminus \Sigma$.
	
	By (\ref{eq-m-integral}), we then have
	\begin{align}\nonumber
	\frac{\partial}{\partial t}&\(2(n-1)m\omega_{n-1}+\int_{\Sigma}fHd\mu_t\)\\
	&\leq \frac{n-2}{n-1}\(\int_\Sigma  fH d\mu_t+2(n-1)m\omega_{n-1}\)-2\int_{\widehat{\partial M}}\nabla f\cdot\nu\,d\mu\label{eq-evo2}\\
	&\leq \frac{n-2}{n-1}\(\int_\Sigma  fH d\mu_t+2(n-1)m\omega_{n-1}\)\,d\mu,\nonumber
	\end{align} 
	where the last inequality follows by the Hopf lemma and Theorem \ref{thm-fzero}, and is in fact strict unless $\widehat{\partial M}$ is empty. This then implies that $Q(t)$ is strictly monotonically decreasing, unless each $\Sigma_t$ is umbilic and $\widehat{\partial M}$ is empty, in which case $Q(t)$ remains constant.
\end{proof}

We now turn to discuss the weak flow. A solution of weak IMCF is generally defined by the level sets of a function $u\geq 0$ on $M$. For each $t>0$, $\Sigma_t:=\partial\{u< t\}$ defines an expanding family of $C^{1,\alpha}$ hypersurfaces that minimise area among homologous hypersurfaces in the region $\{u\geq t\}$. As we do not require technical details of the flow in this short note, we omit further discussion and refer the interested reader to \cite{HI01}.

It is straightforward to verify that the analysis in Section 4 of \cite{Wei2017} depends only on the fact that $M$ admits a positive, bounded static potential $f$. In particular, if there are no boundary components in the region between $\Sigma_{t_1}$ and $\Sigma_{t_2}$, we have (cf. equation (4.12) of \cite{Wei2017})
\begin{equation}\label{s2:eq2}
\int_{\Sigma_{t_2}}fHd\mu_{t_2}-\int_{\Sigma_{t_1}}fHd\mu_{t_1}~\leq~\int_{t_1}^{t_2}\int_{\Sigma_s}\left(2\nabla f\cdot\nu+\frac{n-2}{n-1}fH\right)d\mu_sds.
\end{equation}
Now, each $\Sigma_s$ can be approximated in $C^{1,\alpha}$, so the first integral on the right hand side can be estimated by
\begin{equation*}
\int_{\Sigma_s}\nabla f\cdot\nu\leq (n-2)\omega_{n-1}m
\end{equation*}
with equality if and only if $M$ has no boundary components outside of $\Sigma_s$. Hence
\begin{equation}\label{s2:eq3}
\int_{\Sigma_{t_2}}fHd\mu_{t_2}-\int_{\Sigma_{t_1}}fHd\mu_{t_1} \leq \frac{n-2}{n-1}\int_{t_1}^{t_2}\(\int_{\Sigma_s}\left(fH\right)d\mu_s+2(n-1)\omega_{n-1}m\)ds.
\end{equation}
As $\Sigma$ is outer-minimizing, under IMCF we have that $|\Sigma_t|=e^t|\Sigma|$, and it therefore follows from Gronwall's Lemma -- as in Section 4 of \cite{Wei2017} -- that $Q(t)$ is strictly monotonically decreasing, unless each $\Sigma_t$ is umbilic and $\widehat{\partial M}$ has no boundary components outside $\Sigma_t$. In the case where the flow gets close to another boundary component of $M$, the fact that the other boundary components are minimal surfaces ensure that the analysis is unchanged at the jump. That is, when the flow jumps across another boundary component it behaves identically as to when it jumps across the horizon in the Schwarzschild case, and therefore $Q(t)$ remains monotone. That is, we have monotonicity of $Q(t)$ along the weak flow (cf. Proposition 4.5 of \cite{Wei2017}).

It remains to be shown that $Q(t)$ has the correct limiting behaviour at as the flow runs out to infinity. However, this too follows from the analysis in \cite{Wei2017}. One easily checks that the proof of Proposition 5.1 in \cite{Wei2017} is valid for a general static manifold (cf. equation (5.6) therein) yielding 

\be 
\lim_{t\rightarrow\infty}Q(t)=(n-1)\omega_{n-1}^{\frac{1}{n-1}}.
\ee

\begin{proof}[Proof of Theorem \ref{thm-main}]
	
	Since the weak IMCF runs from the outer-minimizing surface $\Sigma$ out to infinity, the main theorem is then an immediate consequence of the monotonicity combined with the limiting behaviour of $Q(t)$. That is, we have
	
	\be 
	|\Sigma|^{-\frac{n-2}{n-1}}\( 2(n-1)\omega_{n-1}m+\int_\Sigma fH dS \)\geq (n-1)\omega_{n-1}^{\frac{1}{n-1}},
	\ee 
	which can then be expressed as
	\be 
	\frac{1}{(n-1)\omega_{n-1}}\int_\Sigma fH dS\geq \( \frac{|\Sigma|}{\omega_{n-1}} \)^{\frac{n-2}{n-1}} -2m.
	\ee 
\end{proof}

\end{document}